\documentclass[reqno]{amsart}

\usepackage{amsmath,amsfonts,amssymb,amscd,verbatim,delarray}

\usepackage{amsthm}

\usepackage{url}

\newcommand{\C}{\mathcal{C}}
\newcommand{\D}{\mathcal{D}}

\newcommand{\V}{\mathcal{V}}
\newcommand{\G}{\Gamma}

\newcommand{\Th}{\Theta}
\newcommand{\g}{\gamma}
\renewcommand{\b}{\beta}
 \newcommand{\dl}{\delta}

\newcommand{\om}{\omega}
\newcommand{\Om}{\Omega}

\newcommand{\Q}{\mathbf{Q}}
\DeclareMathOperator{\End}{End}
\DeclareMathOperator{\Set}{Set}

\begin{document}

\newtheorem{theorem}{Theorem}
\newtheorem{thm}{Theorem}
\newtheorem*{thm*}{Theorem}
\newtheorem*{conjecture*}{Conjecture}
\newtheorem{cor}{Corollary}
\newtheorem{rem}{Remark}
\newtheorem*{notation}{Notation}
\newtheorem{corollary}[theorem]{Corollary}
\newtheorem{corol}[theorem]{Corollary}
\newtheorem{conj}[theorem]{Conjecture}
\newtheorem{prop}[theorem]{Proposition}
\newtheorem{lemma}[theorem]{Lemma}
\newtheorem{defn}{Definition}
\newtheorem{lem}{Lemma}
\newtheorem{prob}{Problem}
\theoremstyle{definition}
\newtheorem{definition}[theorem]{Definition}
\newtheorem{example}[theorem]{Example}
\newtheorem{remarks}[theorem]{Remarks}
\newtheorem{remark}[theorem]{Remark}
\newtheorem{question}[theorem]{Question}
\newtheorem{problem}[theorem]{Problem}
\newtheorem{questions}[theorem]{Questions}
%%%%%%%%%%%%%%%%%%%%%%%%%%%%%%%%%%%%%%%%%%%%%%%%%%%%%
\def\toeq{{\stackrel{\sim}{\longrightarrow}}}
\def\into{{\hookrightarrow}}

\def\AA{\Bbb A}
\def\BA{\mathbb A}
\def\HH{\mathbb H}
\def\PP{\Bbb P}

%tilde
\def\Xt{{\widetilde X}}
\def\Gt{{\widetilde G}}

\def\hh{{\mathfrak h}}
\def\lie{\mathfrak a}

\def\XX{\mathfrak X}
\def\RR{\mathfrak R}
\def\NN{\mathfrak N}

\def\minus{^{-1}}

\def\GL{\textrm{GL}}            \def\Stab{\textrm{Stab}}
\def\Gal{\textrm{Gal}}          \def\Aut{\textrm{Aut\,}}
\def\Lie{\textrm{Lie\,}}        \def\Ext{\textrm{Ext}}
\def\PSL{\textrm{PSL}}          \def\SL{\textrm{SL}} \def\SU{\textrm{SU}}
\def\loc{\textrm{loc}}
\def\coker{\textrm{coker\,}}    \def\Hom{\textrm{Hom}}
\def\im{\textrm{im\,}}           \def\int{\textrm{int}}
\def\inv{\textrm{inv}}           \def\can{\textrm{can}}
\def\id{\textrm{id}}              \def\Char{\textrm{char}}
\def\Cl{\textrm{Cl}}
\def\Sz{\textrm{Sz}}
\def\ad{\textrm{ad\,}}
\def\SU{\textrm{SU}}
\def\Sp{\textrm{Sp}}
\def\PSL{\textrm{PSL}}
\def\PSU{\textrm{PSU}}
\def\rk{\textrm{rk}}
\def\PGL{\textrm{PGL}}
\def\Ker{\textrm{Ker}}
\def\Ob{\textrm{Ob}}
\def\Var{\textrm{Var}}
\def\poSet{\textrm{poSet}}
\def\Al{\textrm{Al}}
\def\Int{\textrm{Int}}
\def\Smg{\textrm{Smg}}
\def\ISmg{\textrm{ISmg}}
\def\Ass{\textrm{Ass}}
\def\Grp{\textrm{Grp}}
\def\Com{\textrm{Com}}
\def\rank{\textrm{rank}}

\def\char{\textrm{char}}

\def\wid{\textrm{wd}}

\newcommand{\Or}{\operatorname{O}}
\renewcommand{\a}{\alpha}
\newcommand{\thmref}[1]{Theorem~\ref{#1}}
\newcommand{\secref}[1]{Section~\ref{#1}}
\newcommand{\subsecref}[1]{Subsection~\ref{#1}}
\newcommand{\lemref}[1]{Lemma~\ref{#1}}
\newcommand{\corref}[1]{Corollary~\ref{#1}}
\newcommand{\propref}[1]{Proposition~\ref{#1}}

\newtheorem*{examples*}{Examples}

\newcommand{\st}{\sigma}
\renewcommand{\k}{\kappa}
\renewcommand{\th}{\theta}

\def\AGL{{\mathbb G\mathbb L}}
\def\ASL{{\mathbb S\mathbb L}}
\def\ASU{{\mathbb S\mathbb U}}
\def\AU{{\mathbb U}}

\title{On automorphisms of categories}
\author[Grigori Zhitomirski]{G. Zhitomirski\\
 Department of Mathematics and Statistics \\
 Bar-Ilan University, 52900 Ramat Gan, Israel}
\thanks {}

\maketitle

\begin{abstract}
Let $\V$ be a variety of algebras of some type $\Om$. An interest to describing automorphisms of the category $\Th \sp 0 (\V )$ of finitely generated free $\V$-algebras was inspired in connection with development of universal algebraic geometry founded by B. Plotkin. There are a lot of results on this subject. A common method of getting such results was suggested and applied by B. Plotkin and the author (\cite{ZhPlVar,ZhPlUnivAlg}). The method is to find all terms in the language of a given variety which determine such $\Om$-algebras that are isomorphic to a given $\Th \sp 0 (\V )$-algebra  and have the same underlying set with it.  But this method can be applied only to automorphisms which take all objects to isomorphic ones. 

The aim of the present paper is to suggest another method that is free from the mentioned restriction. 
This method is based on two main theorems. Let $\C$ be a category supplied with a faithful  representative functor to the category of sets. Theorem \ref{allAuto} gives a general description of automorphisms of $\C$, using a new notion of a quasi-inner automorphism.  Theorem \ref{product} shows how to obtain the full description of automorphisms of the category $\Th \sp 0 (\V )$. The investigation ends with two examples. The first of them shows the preference of our method in a known situation (the variety of all semigroups) and the second one demonstrates obtaining new results (the variety of all modules over arbitrary ring with unit).

\end{abstract}

\baselineskip 20pt
\bigskip

\section{INTRODUCTION}\label{Intro}

%\subsection{}
Let $\V$ be a variety of algebras of some type $\Om$ and  $\Th (\V)$  be the category of all $\V$-algebras and 
their homomorphisms. For an infinite set $X\sb0$, let $\Th \sp 0 (\V )$ denote the full subcategory of $\Th (\V )$ which is determined by all free $\V$-algebras over finite subsets of the set $X\sb0$. The problem is to describe all automorphisms of the category $\Th \sp 0 (\V )$.  Motivations for this research can be found in the papers \cite{MPP, MPP2, Seven, AlgGeom}. The most important case is when automorphisms of this category are inner or close to inner in a sense.

An automorphism $\Phi$ of a category $\C$ is called inner if it is isomorphic to the identity functor $Id \sp
{\C}$   in the category of all endofunctors of $\C$. It means that for every object $A$ of the given category
there exists an isomorphism $\sigma\sb A : A\to \Phi (A)$ such that for every morphism $\mu: A\to B$ we have $\Phi (\mu)=
\sigma\sb B \circ \mu \circ \sigma\sb A \sp {-1}$.  Thus an automorphism $\Phi$ may be inner only in the case when there is an isomorphism between $A$ and $\Phi (A)$ for every $\C$-object $A$. 

As regards the inverse proposition, it was shown in the papers \cite{ZhPlVar,ZhPlUnivAlg}  that if the objects $\Phi (A)$ and   $A$ are isomorphic for every $\C$-object $A$, then $\Phi$ is inner or so called potentially inner. The last one means that $\Phi $ is inner in a category obtained by adding to $\C$ some new morphisms. More accurately, there exists a family $(s\sb A :A \to \Phi (A))\sb {A \in Ob \C}$ of morphisms of an extended category such that for every morphism $\mu : A\to B$ of the category $\C$ we have $\Phi (\mu)=s\sb B \circ \mu \circ s\sb A \sp {-1}$. That way, the problem is to describe these new morphisms $s\sb A$.  A general method how to do it was suggested in the mentioned papers for categories  $\C =\Th \sp 0 (\V )$ where  $\V$ is a suitable variety. Below is a summary of the essence of this method. 

Consider a $\C$-algebra $A$ free generated by the set $\{x\sb 1,\dots , x\sb n\}$ of variables, where $n$ is greater than arities of all $\Om$-operators. We can assume that algebras $\Phi (A) $ and $A$ have the common base set $|A|$ and the same free generators. Then for every $k$-ary operation symbol $\om \in \Om$ we have two $k$-ary operations on $|A|$: $\om \sb A$ and  $\om \sb {\Phi (A)}$.

It is obvious that every element of the algebra $A$ can be considered as a term in the language corresponding to the algebra $\Phi (A)$ and vice versa.  In other words, every operation $\om \sb A$ is a derived operation in $\Phi (A) $, and vice versa, 
every operation $\om \sb {\Phi (A)}$ is a derived operation in $A$.  Thus the mappings $s\sb A:|A|\to|A|$ such that $s\sb A (x\sb i)=x\sb i$ and $s\sb A (\om \sb A (x\sb 1, \dots ,x\sb k))=\om \sb {\Phi (A)} (x\sb 1, \dots ,x\sb k))$ for every $\om \in \Om$ determine the automorphism $\Phi$. The problem is reduced to fined terms $\tilde{\om }$ in the algebra $\Phi (A)$ such that the derived algebra $(|A|, (\tilde{\om })\sb {\om \in \Om}) $ belongs to the variety $\V$.

Some new results were obtained on this way.  Particularly,  A. Tsurkov successfully applied this method to many-sorted algebras (for example in \cite{Tsurkov}). It is known that if a variety $\V$ satisfies IBN-property then every automorphism of the category $\Th \sp 0(\V ) $ takes every object to an isomorphic one. But there are varieties without IBN-property. If it is unknown whether  $A$ and $\Phi (A) $ are isomorphic or not, the method described above does not work.

The aim of the current investigation is to fill up this gap. Since an automorphism may be not inner in the general case, we use a new type of automorphisms called the {\it quasi-inner automorphisms} (Definition \ref{inner}), and it turns out that this notion is enough to characterize arbitrary automorphisms. Further reasoning as a matter of fact follows the ideas of the papers \cite{ZhPlVar,ZhPlUnivAlg}  when $A$ and $\Phi (A)$ are isomorphic, but instead of the method outlined above, a new method is proposed that to the author's opinion is more successful. 

This method reduces the problem to the case when the underlying set $|A|$ of an algebra $A$ is a subset of the underlying  set of the algebra $\Phi (A)$, and every endomorphism $\mu$ of the algebra $A$ is the restriction of the endomorphism $\Phi (\mu)$ of the algebra $\Phi (A)$ to  the set $|A|$ and every restriction to $|A|$ of an endomorphism $\nu $ of  $\Phi (A) $ is an endomorphism $\mu$ of $A$ such that  $\Phi (\mu)=\nu$. That circumstance gives us an opportunity to describe the action of the automorphism $\Phi$. 

The main results are formulated in two theorems. Theorem \ref{allAuto} states that every automorphism of 
a category $\C$ supplied with a forgetful functor $\Q:\C \to \Set$ which satisfies two acceptable conditions is potentially quasi-inner. Theorem \ref{product} states that every automorphism $\Phi$ of the category $\Th \sp 0 (\V )$ for an arbitrary variety $\V$ is the product of two functors  $\Phi =\G \circ \Psi$. The first of them is an inner isomorphism $\Psi :\Th \sp 0 (\V ) \to \D$ ,where $D$ is a full subcategory of $\Th \sp 0 (\V )$ which is described. The second functor $\G: \D \to \Th \sp 0 (\V )$ is a so called extension functor, that is, $|A| \subseteq |\G (A)|$  for every $\C$-algebra $A$ and $\mu \subseteq \G (\mu) $ for every $\C$-morphism $\mu$.

The last part of the paper contains two examples. The first of them shows the preference of our method in a known situation (the variety of all semigroups) and the second one presents a new result for the variety of all modules over arbitrary ring with unit. As a consequence of the last result we obtain that in the case when the ring does not contain zero divisors all automorphisms are semi-inner. 
%section{Quasi-inner automorphisms}\label{quasi}
\section{Quasi-inner automorphisms}\label{quasi}
We consider only such categories $\C$ which are supplied with a faithful functor  $\Q: \C \to Set$, where $Set$ is the category of all sets. We call $\Q$ the forgetful functor as it is accepted.
\begin{defn}\label{inner}   An automorphism $\Phi$ of a category $\C$ is called inner if there is an isomorphism between $\Phi$ and the identity functor $Id \sp{\C}$  in the category of all endofunctors of $\C$. An automorphism $\Phi$ is said to be {\it quasi-inner} if there is a bimorphism  $Id \sp{\C}\to \Phi$ . 
\end{defn}
This definition means that $\Phi$ is inner if for every object $A$ of the given category there exists an isomorphism $\sigma\sb A : A\to \Phi (A)$ such that for every morphism $\mu: A\to B$ the following diagram commutes:
\[
\CD
A @> \sigma \sb A>> \Phi (A)\\
@V\mu VV       @VV\Phi (\mu) V \\ 
B @> \sigma \sb B >> \Phi (B)
\endCD
\]
Hence we get that $\Phi (\mu)=\sigma\sb B \circ \mu \circ \sigma\sb A \sp {-1}$. 

As a generalization of this notion, $\Phi$ is quasi-inner if for every object $A$ of the given category there exists a {\it bimorphism} $\sigma\sb A : A\to \Phi (A)$ such that for every morphism $\mu: A\to B$ the diagram above is commutative. Hence we have only the following condition $\sigma\sb B \circ\mu =\Phi (\mu) \circ \sigma\sb A $  but not the expression for $\Phi (\mu)$ like in the case of inner automorphism if the morphism $\sigma\sb A $ is not invertible. Nevertheless, we have $\Q (\mu)=(\Q (\sigma\sb B ))\sp {-1}\circ \Q (\Phi (\mu)) \circ \Q (\sigma\sb A) $.
\begin{lem}\label{square}
Let $\Phi$  be a quasi-inner automorphism of $\C$ and $(\sigma\sb A)_{A\in Ob\C}:Id^\C \to \Phi$ be a bimorphism. If  $\sigma\sb B \circ \mu =\nu \circ \sigma\sb A $ for $\mu:A\to B, \; \nu : \Phi (A) \to \Phi (B) $,  then $\nu=\Phi (\mu)$.
Hence $\Phi (\mu)$ is uniquely determined  by the following inclusion: $\Q (\sigma\sb B )\circ \Q (\mu) \circ (Q (\sigma\sb A))\sp {-1} \subseteq \Q (\Phi (\mu)) $.
\end{lem}
\begin{proof}
It is obvious because $\sigma\sb A $ is an epimorphism. 
\end{proof}

%%   Section Potentially quasi-inner automorphisms
\section{potentially quasi-inner automorphisms}\label{PotQuasi}
In this section, we assume that every category $\C$ under consideration is supplied with a forgetful functor  $\Q$ such that there exist a $\C$-object $A\sb 0$ and an element $x\sb 0 \in \Q (A\sb 0)$ which satisfy the following conditions:
\begin{enumerate}
\item[1Q)] $\Q$ is represented by the pair $(A\sb 0 ,x\sb 0)$, i.e., for every object $A$ of $\C$ and for every element $a\in \Q(A)$ there is exactly one morphism $\alpha:A\sb 0 \to A$ such that $ \Q(\alpha) (x\sb 0)=a$.
\item[2Q)] for every $\C$-object $A$  there exists a morphism $\a:A\to A\sb 0$ such that for every element $x\in \Q (A\sb 0 )$ we have $x=\Q (\a) (a)$ for some element $a\in \Q(A)$, in other words,  $\Q (\a):\Q (A)\to \Q (A\sb 0)$ is surjective.
\end{enumerate}
Some of the simple properties of such categories need to be noted:
\begin{lem}\label{mono_and_epi}
%\begin{enumerate}
\item[1.] Let $\mu: A\to B$ be a $\C$-morphism. If $\Q(\mu):\Q(A) \to \Q(B)$ is surjective then $\mu$ is an epimorphism.
\item[2.] A morphism  $\mu: A\to B$ is a monomorphism if and only if the map $\Q(\mu):\Q(A) \to \Q(B)$ is injective. 
\item[3.] Let $\Phi$ be an automorphism of the category $\C$. There exists an epimorphism $\eta :A\sb 0 \to \Phi (A\sb 0)$. If  $\eta$ is an isomorphism then $\Phi$ is potentially inner.
%\end{enumerate}
\end{lem}
\begin{proof}
\item[1.] It is obvious.
\item[2.] It is also obvious that if $\Q (\mu)$ is injective then $\mu$ is a monomorphism. Let $\mu: A\to B$ is a monomorphism. Suppose that $\Q (\mu)$ is not injective, i. e., there exist elements $a\sb 1, a\sb 2 \in \Q (A)$ such that $a\sb 1 \not =a\sb 2 $ but $\Q (\mu) (a\sb 1)=\Q (\mu) (a\sb 2)$. According to property 1Q, there are morphisms $\a\sb 1, \a\sb 2 :A\sb 0 \to A$ such that $\Q (\a\sb 1) (x\sb 0) =a\sb 1$ and  $\Q (\a\sb 2) (x\sb 0) =a\sb 2$. We have  $\Q(\mu) ( \Q(\a\sb 1) (x\sb 0))=\Q (\mu) ( \Q(\a\sb 2) (x\sb 0))$ and hence $\mu\circ \a\sb 1=\mu\circ \a\sb 2$. Since $\mu$ is a monomorphism we obtain that $\a\sb 1 =\a\sb 2$ which contradicts the supposition $a\sb 1 \not =a\sb 2 $ . 
\item[3.] Consider the object $\Phi \sp {-1} (A\sb 0)$. According to the property Q2 there is a morphism $\eta \sb 0 : \Phi \sp {-1} (A\sb 0) \to A\sb 0$ such that the mapping $\Q (\eta \sb 0) :\Q ( \Phi \sp {-1} (A\sb 0)) \to \Q (A\sb 0)$ is surjective. Thus 
$\eta \sb 0 : \Phi \sp {-1} (A\sb 0) \to A\sb 0$ is an epimorphism. Hence $\eta =\Phi (\eta \sb 0) : A\sb 0 \to \Phi (A\sb 0)$ is an epimorphism too. If $\eta $ is an isomorphism, then $\Phi$ is potentially inner according to Theorem 1 in \cite{ZhPlVar}.
\end{proof}
\begin{defn}\label{main_epi}
The surjective morphism $\eta \sb 0 : \Phi \sp {-1} (A\sb 0) \to A\sb 0$ and the epimorphism $\eta :A\sb 0 \to \Phi (A\sb 0)$ introduced by the proof of \lemref{mono_and_epi} are fixed as main epimorphisms connected with $\Phi$ and are denoted by $\eta \sb 0\sp \Phi$ and $\eta \sp \Phi$ correspondingly. If $A\sb 0 = \Phi (A\sb 0)$ we assume that $ \eta \sp \Phi= \eta \sb 0 \sp \Phi = 1\sb {A\sb 0} $.
\end{defn}
\begin{defn}\label{bijection}
Let $A,B\in Ob\C$. A mapping  $s:\Q(A)\to Q(B)$ is called $\C$-bijection if $s$ is injective and for any two $\C$-morphisms $\a\sb 1, \a\sb 2:  B\to C$ the equality $\Q(\a\sb 1)\circ s =\Q(\a\sb 2)\circ s$ implies $\a\sb 1=\a\sb 2$. Roughly speaking the mapping $s$ has the epimorphism property only in relation to $\C$-morphisms. 
\end{defn}
It is obvious that on the one hand the notion of $\C$-bijection is a generalization of the notion of bijection and, on the other
hand, if for a $\C$-morphism $\st :A\to B$ the mapping $\Q(\st):\Q(A)\to Q(B)$ is a $\C$-bijection then $\st :A\to B$ is a bimorphism.
\begin{defn}\label{potentially}
An automorphism $\Phi$ of a category $\C$ is called potentially quasi-inner if there exists a family of $\C$-bijections $(s\sb A :\Q(A) \to \Q(\Phi (A)))\sb {A\in Ob\C}$   such that for every $\C$-morphism $\mu :A\to B$ the following diagram commutes:
$$
\CD 
\Q (A) @> s\sb A>>\Q (\Phi (A))\\
@V\Q(\mu )VV        @VV\Q(\Phi (\mu)) V \\
\Q(B) @> s\sb B >>\Q ( \Phi (B))
\endCD
$$
It is easy to see that for potentially quasi-inner automorphisms the fact analogous to\lemref{square} is true too, i. e., if $s\sb B \circ \Q(\mu) =\Q (\nu) \circ s\sb A $ for some morphism  $\nu : \Phi (A) \to  \Phi (B)$ then $\nu =\Phi (\mu)$. Thus the morphism $\Phi (\mu)$ is uniquely determined by the commutativity of the diagram above.
\end{defn}
\begin{defn}\label{alpha}
Let $A$ be a $\C$-object and $a\in \Q (A)$.  We denote by $\a \sb a \sp A$ the unique morphism $A\sb 0 \to A$ such that 
$\Q (\a \sb a \sp A) (x\sb 0)=a$. 
\end{defn}
Now we are ready to prove the first of two main theorems.
\begin{thm}\label{allAuto}
Every automorphism of a category which satisfies conditions 1Q and 2Q  is potentially quasi-inner.
\end{thm}
\begin{proof}
Let $\Phi$ be an automorphism of a category $\C$  and $A$ be a $\C$-object.  Define the mapping $s\sb A :\Q (A) \to \Q (\Phi (A))$ by the formula:
\begin{equation}\label{s}
s\sb A (a)=\Q(\Phi (\a \sb a \sp A)\circ \eta ) (x\sb 0)
\end{equation}
for every $a\in A$. Here $\eta =\eta \sp \Phi $ according to Definition \ref{main_epi}.

Start to check that $s\sb A$ is a $\C$-bijection. 
\item{1)}. Let $s\sb A (u)=s\sb A (v)$ for $u,v\in \Q(A)$. Then $\Q (\Phi (\a \sb u \sp A)\circ \eta ) (x\sb 0) =\Q (\Phi (\a \sb v \sp A)\circ \eta )(x\sb 0)$. Hence $\Phi (\a \sb u \sp A)\circ \eta  =\Phi (\a \sb v \sp A)\circ \eta $. Since $\eta $ is a $\C$-epimorphism we obtain that $\Phi (\a \sb u \sp A)=\Phi (\a \sb v \sp A)$ which implies $\a \sb u \sp A =\a \sb v \sp A$, i. e., $u=v$. Hence $s\sb A$ is injective. 
\item{2)}.
Let $\Q (\g)\circ s\sb A =\Q (\dl) \circ s\sb A$ for some $\C$-morphisms $\g , \dl :\Phi (A) \to B$. Using \ref{s} we have
for all $a\in \Q (A)$
$$\Q (\g)\circ \Q(\Phi (\a \sb a \sp A)\circ \eta)(x\sb 0)=\Q (\dl) \circ \Q(\Phi (\a \sb a \sp A)\circ \eta)(x\sb 0)$$
and hence 
$$ \g\circ \Phi (\a \sb a \sp A)\circ \eta=\dl\circ \Phi (\a \sb a \sp A)\circ \eta $$
for all $a\in \Q (A)$. 

Since $\eta$ is an epimorphism we obtain $\g\circ \Phi (\a \sb a \sp A)=\dl\circ \Phi (\a \sb a \sp A)$. Applying $\Phi\sp {-1}$ to both sides of this equation we get that $\Phi\sp {-1}(\g)\circ \a \sb a \sp A =\Phi\sp {-1}(\dl)\circ \a \sb a \sp A$ for all $a\in \Q (A)$. Hence $\Phi\sp {-1}(\g) =\Phi\sp {-1}(\dl)$ which gives $\g =\dl$. Thus $s\sb A$ is $\C$-bijective. 

It remains to check that the corresponding square in Definition \ref{potentially} is commutative.  According to \ref{s} we have for every $a\in \Q (A)$ 
\begin{gather*}
\Q (\Phi (\mu))\circ s\sb A (a) =\Q (\Phi (\mu))\circ  \Q(\Phi (\a \sb a \sp A)\circ \eta)(x\sb 0)=\\
=\Q(\Phi (\mu)\circ \Phi (\a \sb a \sp A)\circ \eta)(x\sb 0) =\Q(\Phi (\mu \circ \a \sb a \sp A) \circ \eta)(x\sb 0) =\\
=\Q(\Phi (\a \sb {\Q (\mu) (a)} \sp B) \circ \eta)(x\sb 0) =s\sb B (\Q (\mu ) (a))=s\sb B\circ \Q (\mu) (a). 
\end{gather*}
Thus $\Q (\Phi (\mu))\circ s\sb A =s\sb B \circ \Q (\mu)$. 
\end{proof}
\begin{defn}\label{main_function}
The family of $\C$-bijections $(s\sb A :\Q(A) \to \Q(\Phi (A)))\sb {A\in Ob\C}$ defined by \ref{s} we call the {\bf main function} corresponding to $\Phi$. 
\end{defn}
It is easy to see that this notion is a generalization of the notion of the main function defined in  \cite{ZhPlVar}. In fact, if the main epimorphism  $\eta : A\sb 0 \to \Phi (A\sb 0)$ is an isomorphism we get the same formula for $s\sb A$ (see \cite{ZhPlVar} formula 2.2). 

Of course the main function is not the unique function $(\Q(A) \to \Q(\Phi (A)))\sb {A\in Ob\C}$ which makes the squares in  Definition \ref{potentially} commutative. For example, the main function for the identity functor is equal (according to \ref{s} and Definition \ref{main_epi})  to $s\sb A (a)=\Q (\a \sb a \sp A\circ 1\sb {A\sb 0} ) (x\sb 0)=a$ , that is, $s\sb A =\Q (1\sb A)$. Hence in that case we have $ s\sb B \circ \Q (\mu ) =\Q (\mu )\circ s\sb A$ for every morphism $\mu :A\to B$.  
Having in mind this property we use the following notion (see \cite{ZhPlVar,ZhPlUnivAlg}). 
\begin{defn}\label{cntr}
The family $(c\sb A)\sb {A\in Ob\C}$ of $\C$-bijective mappings $c\sb A :\Q(A) \to \Q(A)$ having the following property: $c\sb B \circ \Q (\mu )=\Q (\mu )\circ c\sb A$ for every morphism $\mu :A\to B$ is called a {\it central function}. In other words, a function $A \to c\sb A $ is a central function if it determines the identity automorphism of the category $\C$. 
\end{defn}

Let $\Phi$ be an automorphism of $\C$.   \thmref{allAuto} states that both $\Phi$ and $\Phi \sp{-1}$ are potentially quasi-inner.
Therefore we have two main functions  $(s\sp {\Phi}\sb A :A \to \Phi (A))\sb {A\in Ob\C}$ and $(s\sp {\Phi \sp {-1}}\sb A :A \to \Phi \sp {-1} (A))\sb {A\in Ob\C}$ of $\C$-bijections and two main epimorphisms $\eta \sp {\Phi}: A\sb 0\to \Phi (A\sb 0)$ and $\eta \sp {\Phi \sp {-1}}: A\sb 0\to \Phi \sp{-1}  (A\sb 0)$. 

Although in general case, the main functions do not satisfy many conditions which are valid in the case $\eta$ is an isomorphism, they have some similar properties. In particular,  we get   the following commutative diagram for every morphism $\mu :A\to B$:
$$
\CD 
\Q (A) @> s\sp {\Phi}\sb A>>\Q ( \Phi (A))@>s\sp {\Phi \sp {-1}}\sb {\Phi (A)}>>\Q (A)\\
@V\Q(\mu )VV        @VV\Q(\Phi (\mu)) V   @VV\Q(\mu )V\\
\Q (B) @>s \sp {\Phi}\sb B >>\Q (\Phi (B))@>s\sp {\Phi \sp {-1}}\sb {\Phi (B)}>>\Q(B)
\endCD
$$
We see that the family of mappings  $c\sb A \sp {\Phi}=s\sp {\Phi \sp {-1}}\sb {\Phi (A)}\circ s\sp {\Phi}\sb A:\Q (A)\to \Q (A)$ is a central function. This central function will be called the {\it main central function} for the automorphism $\Phi$.

It is obvious that for every function $A\mapsto s\sb A $ that determines an automorphism $\Phi$ and any two central functions $A \mapsto e\sb A$ and $A\mapsto d\sb A$ the function $A \mapsto (e\sb {\Phi (A)} \circ s\sb A \circ d\sb A )$ determines the same automorphism. The converse proposition is also true.
\begin{prop}\label{twofunctions}
Let a function $A\mapsto t\sb A$, where $A\in Ob\C$, determines an automorphism $\Phi$ of a category $\C$. Then there exists a central function $A\mapsto d\sb A$ such that $t\sb A =(c\sb{\Phi (A)}\sp {\Phi \sp{-1}}) \sp {-1}\circ s\sp {\Phi}\sb A \circ d\sb A $ for all $\C$-objects $A$.
\end{prop}
\begin{proof}
Let $d\sb A = s\sb {\Phi (A)}\sp {\Phi \sp {-1}}\circ t\sb A$. It is obvious that the function $A\mapsto d\sb A$ is central. 
We get $s\sb{\Phi (A)}\sp {\Phi} \circ s\sb {\Phi (A)}\sp {\Phi \sp {-1}}\circ t\sb A=s\sb {\Phi (A)}\sp {\Phi} \circ d\sb A$.
Since $s\sb{\Phi (A)}\sp {\Phi} \circ s\sb {\Phi (A)}\sp {\Phi \sp {-1}}=c\sb{\Phi (A)}\sp {\Phi \sp{-1}}$ is injective, we get 
$t\sb A =(c\sb{\Phi (A)}\sp {\Phi \sp{-1}}) \sp {-1}\circ s\sp {\Phi}\sb A \circ d\sb A $
\end{proof}

%% Section The categories of free  universal algebras
\section{The categories of free universal algebras}\label{Free}
From now until the end of the article we deal with an arbitrary variety $\V$  of universal algebras of some signature $\Om$ and with the corresponding category $\Th (\V)\sp 0$  of all finitely generated $\V$-algebras and their homomorphisms. This category and all its subcategories are supplied with usual forgetful functor $\Q$ that assigns to every algebra $A$ its underlying set $|A|$ and to every homomorphism $A\to B$ the corresponding mapping $|A|\to |B|$. If there are no misunderstanding we denote an algebra by the same letter as its underlying set  and do  the same  for homomorphisms.  

Let  $A\sb 0$ be a monogenic  algebra free generated by an element $x\sb 0$. We consider such full subcategories of $\Th (\V)\sp 0$ that contain the object  $A\sb 0$.  It is obvious that every such category satisfies the conditions 1Q and 2Q formulated in \secref{PotQuasi} because the pair $(A\sb 0, x\sb 0)$ represents the forgetful functor $\Q$.

Let $\C$ be a full subcategory of $\Th (\V)\sp 0$  and $A\sb 0$ be an object of  $\C$. Let $\Phi$ be an automorphism of $\C$.   \thmref{allAuto} states that both $\Phi$ and $\Phi \sp{-1}$ are potentially quasi-inner.
Therefore we have two main functions  $(s\sp {\Phi}\sb A :A \to \Phi (A))\sb {A\in Ob\C}$ and $(s\sp {\Phi \sp {-1}}\sb A :A \to \Phi \sp {-1} (A))\sb {A\in Ob\C}$ of $\C$-bijections and four main epimorphisms:
$$\eta \sp {\Phi}\sb 0: \Phi \sp {-1} (A\sb 0) \to A\sb 0,\; \eta \sp {\Phi}: A\sb 0\to \Phi (A\sb 0),$$
 $$\eta \sb 0\sp {\Phi \sp {-1}}: \Phi (A\sb 0) \to \ A\sb 0,\;   
 \eta \sp {\Phi \sp {-1}}: A\sb 0\to \Phi \sp{-1}  (A\sb 0).$$ 

%%Now it is worth to define concretely the main epimorphism $\eta \sp {\Phi}: A\sb 0\to \Phi (A\sb 0)$.
By definition, $\eta \sp {\Phi} =\Phi (\eta \sp {\Phi}\sb 0)$. The epimorphism $\eta \sp {\Phi}\sb 0$ was introduced in the process of the proof of \lemref{mono_and_epi} as an epimorphism $\Phi \sp {-1} (A\sb 0) \to A\sb 0$ but now we can specify it. Let $X$ be a basis of $\Phi \sp {-1} (A\sb 0)$. Then we set $\eta \sp {\Phi}\sb 0 (x) =x\sb 0$ for all $x\in X$. It is obvious that some arbitrariness remains in the definition of  $\eta \sp {\Phi}\sb 0$  which depends on the choice of  $X$. Specifically, if there is an isomorphism between $\Phi \sp {-1} (A\sb 0)$ and $A\sb 0$ we prefer to assume that $\eta \sp {\Phi}\sb 0$ is equal to corresponding isomorphism. All this concerns the epimorphism $\eta \sb 0\sp {\Phi \sp {-1}}$ too.

Just as in the previous section, we have the main central function $(c\sb A\sp{\Phi})\sb {A\in Ob\C}$.  Now we describe the images of this function.
\begin{prop}\label{c-work}
Let $A$ be a $\C$-algebra and $a\in A$. Let $c\sb A\sp {\Phi} =s\sp {\Phi \sp {-1}}\sb {\Phi (A)}\circ s\sp {\Phi}\sb A$. Denote by $w\sp {\Phi}(x\sb 0)$ the term $\eta \sp {\Phi} \sb 0\circ \eta \sp {\Phi \sp {-1}}  (x\sb 0)\in A\sb 0$. 
Then $c\sb A\sp {\Phi} (a)$ is an element of subalgebra of $A$ generated by $a$, namely, $ c\sb A\sp {\Phi} (a)=w\sp {\Phi}(a)$. If $w\sp {\Phi}(x\sb 0)=x\sb 0$  then all mappings $c\sb A \sp {\Phi}$ are the identity mappings and therefore $\Phi$ is a potentially inner automorphism.
\end{prop}
\begin{proof}We have
$$c\sb A \sp {\Phi}(a)=s\sp {\Phi \sp {-1}}\sb {\Phi (A)}(s\sp {\Phi}\sb A (a))=\Phi \sp {-1}(\a \sp {\Phi (A)}\sb {s\sp {\Phi}\sb A (a)})\circ \eta \sp {\Phi \sp {-1}}(x\sb 0) =\Phi \sp {-1}(\a \sp {\Phi (A)}\sb {s\sp {\Phi}\sb A (a)}\circ \eta \sp {\Phi \sp {-1}}\sb 0) (x\sb 0).
$$
Now let calculate the mapping that is in parentheses. Let $x\sb 1, ... ,x\sb n$ be a basis of $\Phi (A)$. By definition $\eta \sp {\Phi \sp {-1}}\sb 0 (x\sb i )=x\sb 0$. Thus
\begin{gather*}
 \a \sp {\Phi (A)}\sb {s\sp {\Phi}\sb A (a)}\circ \eta \sp {\Phi \sp {-1}}\sb 0 (x\sb i)= \a \sp {\Phi (A)}\sb {s\sp {\Phi}\sb A (a)}(x\sb 0)=s\sp {\Phi}\sb A (a) =\Phi (\a\sp A \sb a)\circ \eta \sp {\Phi} (x\sb 0)=\\
=\Phi (\a\sp A \sb a \circ \eta \sp {\Phi} \sb 0) (x\sb 0)=\Phi (\a\sp A \sb a \circ \eta \sp {\Phi} \sb 0)\circ\eta \sp {\Phi \sp {-1}}\sb 0 (x\sb i).
\end{gather*}
Since $x\sb i$ is an arbitrary element of basis we obtain that 
$$\a \sp {\Phi (A)}\sb {s\sp {\Phi}\sb A (a)}\circ \eta \sp {\Phi \sp {-1}}\sb 0 =\Phi (\a\sp A \sb a \circ \eta \sp {\Phi} \sb 0)\circ\eta \sp {\Phi \sp {-1}}\sb 0.
$$
If we substitute the mapping in parentheses by the obtained value we get that
\begin{gather*}
 c\sb A \sp {\Phi}(a)=\Phi \sp {-1}(\Phi (\a\sp A \sb a \circ \eta \sp {\Phi} \sb 0)\circ\eta \sp {\Phi \sp {-1}}\sb 0) (x\sb 0)=\\
=\a\sp A \sb a \circ \eta \sp {\Phi} \sb 0\circ \Phi \sp {-1}(\eta \sp {\Phi \sp {-1}} \sb 0) (x\sb 0)=\a\sp A \sb a \circ \eta \sp {\Phi} \sb 0\circ \eta \sp {\Phi \sp {-1}}  (x\sb 0) =\a\sp A \sb a (w\sp {\Phi}(x\sb 0))=w\sp {\Phi}(a). 
\end{gather*}
Thus $c\sb A \sp {\Phi}(a)\in \langle a\rangle$. 

If all mappings $c\sb A \sp {\Phi}$ are identity mappings then all $s\sp {\Phi}\sb A$ are bijections and therefore $\Phi$ is a potentially inner automorphism.
\end{proof}

\begin{prop}
Let $A$ be a $\C$-algebra and $X$ be a basis of $A$. For any central function $(c\sb A)\sb {A\in Ob\C}$, either none of elements of $X$ belongs to $c\sb A(A)$ or  $c\sb A$ is surjective. 
\end{prop}
\begin{proof}
Let $x \in c\sb A(A)$ for some $x\in X$ and $a\in A$. Let $\mu$ be an endomorphism of $A$ such that $\mu (x) =a$. Then $a\in \mu (c\sb A (A))$.  Since $c\sb A$ is central we have $a\in c\sb A (\mu (A))$ that implies $a\in c\sb A (A)$. Thus $c\sb A$ is surjective because $a$ is an arbitrary element of $A$.
\end{proof}

We show that the same fact takes place for $s\sp {\Phi} \sb A$. Hereinafter we write $s\sb A$ instead of $s\sp {\Phi}\sb A$ if it is clear what we have in mind.

\begin{prop}\label{notsurjective}
Let $A$ be a $\C$-algebra and $X$ be a basis of $\Phi (A)$. Then either none of elements of $X$ belongs to $s\sb A (A)$ or  $s\sb A$ is surjective. 
\end{prop}
\begin{proof}
Let $x\in s\sb A (A)$ for some  $x\in X$. Let $u$ be an arbitrary element of $\Phi (A)$. There is an endomorphism $\mu$ of 
 $\Phi (A)$ such that $\mu (x)=u$. We have the equation $s\sb A \circ \Phi \sp{-1} (\mu) =\mu  \circ s\sb A $. Since $x\in s\sb A (A)$ there is $a\in A$ such that $x=s\sb A (a)$. We have $s\sb A ( \Phi \sp{-1} (\mu) (a))=\mu (s\sb A (a)) =\mu  (x)=u$. Thus $u\in s\sb A (A)$. Since $u$ is an arbitrary element of $\Phi (A)$, $s\sb A $ is surjective.
\end{proof}

Let $X=\{x\sb 1,\dots ,x\sb n\}$ be a basis of $\C$-algebra $A$. Let $y\sb i=s\sb A\sp {\Phi }(x\sb i)$ for $i=1,\dots , n$. 
Since $s\sb A \sp {\Phi }$ is injective all elements of the set $Y=\{y\sb 1,\dots ,y\sb n\}$ are pairwise  different. One the other hand, \propref{notsurjective} declares that none of elements of $Y$  belongs to any basis of $\Phi (A)$.  Nevertheless
the set $Y$ has the following interesting property. 
\begin{prop}\label{uniq}
Let $B$ be a $\C$-algebra and $\b, \g :\Phi (A) \to B$ be two homomorphisms.  If $\b (y\sb i)=\g (y\sb i)$ for all $i=1,\dots , n$ then $\b =\g$.
\end{prop} 
\begin{proof}
We have $\b (y\sb i)=\b (s\sb A (x\sb i)) =\b (\Phi (\a \sb {x\sb i})\circ \eta (x\sb 0))=\b \circ \Phi (\a \sb {x\sb i})\circ \eta (x\sb 0)$. The same we have for $\g$: $\g (y\sb i)=\g \circ \Phi (\a \sb {x\sb i})\circ \eta (x\sb 0)$. Under the suggestion, we have 
$$\b \circ \Phi (\a \sb {x\sb i})\circ \eta (x\sb 0)=\g \circ \Phi (\a \sb {x\sb i})\circ \eta (x\sb 0),
$$
which implies 
$$\b \circ \Phi (\a \sb {x\sb i}) =\g \circ \Phi (\a \sb {x\sb i})
$$
and then 
$$\Phi (\Phi {-1}(\b) \circ \a \sb {x\sb i}) =\Phi (\Phi {-1}(\g) \circ \a \sb {x\sb i}).
$$
Since $X$ is a basis of $A$ then after some obvious steps we obtain  that $\b =\g $.
\end{proof}
  
Now our aim is to describe images of the main function, that is, the sets $ s\sb A\sp {\Phi} (A)$ for every object $A$. 
\begin{thm}\label{image}
There exists a term ${\bf t}(x\sb 0) $ containing only one variable $x\sb 0$ such that for every $\C$-algebra $A$ the following 
expression takes place:
$$ s\sb A \sp {\Phi} (A) =\{{\bf t}(u) \vert u\in \Phi (A)\}.$$
\end{thm}  
\begin{proof}
Let $x$ be a member of some basis of the algebra $\Phi  \sp {-1} (A\sb 0)$. Then $s\sb {\Phi  \sp {-1} (A\sb 0)}\sp {\Phi}(x)$ is an element of $A\sb 0$ which we consider as a term ${\bf t}(x\sb 0) $. Let $a\in A$ and $\mu :\Phi  \sp {-1} (A\sb 0) \to A$  be a homomorphism such that $\mu (x) = a$.
We have the following commutative diagram:
$$
\CD 
\Phi  \sp {-1} (A\sb 0) @> s\sb {\Phi  \sp {-1} (A\sb 0)}\sp {\Phi}>>A\sb 0\\
@V\mu VV        @VV\Phi (\mu) V \\
A @> s\sb A\sp {\Phi} >> \Phi (A)
\endCD
$$
Hence $s\sb A\sp {\Phi}(a)=\Phi (\mu) ( s\sb {\Phi  \sp {-1} (A\sb 0)}\sp {\Phi}(x))=\Phi (\mu) ({\bf t}(x\sb 0)) ={\bf t}(\Phi (\mu) (x\sb 0))$. 
Thus there exists an element $u\in \Phi (A)$ such that $s\sb A\sp {\Phi} (a)={\bf t}(u)$. It is obvious that  the element ${\bf t}(u)$ belongs to $s\sb A \sp {\Phi}(A)$ for every $u\in \Phi (A)$.
\end{proof}

%Full description of automorphisms of categories of free algebras$
\section{Full description of automorphisms of categories of free algebras.}\label{description}
The object of this section is to give a description of automorphisms of arbitrary categories which are subcategories of the category $\Th \sp 0 ({\V})$ and which contain the object $A\sb 0 $. To this end,  we select two kinds of functors.
Because of the fact that we may have two different algebraic structures on the same set it makes sense to use different notations for an algebra and for its underlying set, namely,  $A$ and $|A|$  correspondingly. 
\begin{defn}\label{functors}
Let $\C$ and $\D$ be full subcategories of the category $\Th \sp 0 ({\V})$. 
\item{1.} An isomorphism $\Psi :\C \to \D$ is called inner if there is a family of isomorphisms $(\st\sb {A} :A\to \Psi (A))\sb {A\in Ob \C}$ such that $\Psi (\mu) =\st\sb B \circ \mu \circ \st\sb A \sp {-1}$ for every morphism $\mu :A\to B$ of the category $\C$. That is,  $\Psi$ is isomorphic to the embedding functor $Id\sb {\C} \sp{\Th \sp 0 ({\V})} :\C \to \Th \sp 0 ({\V})$.
\item{2.} An injective functor $\G :\C \to \D$ is called an extension functor if 
\item{(i)} $|A| \subseteq |\G (A)|$ for every $\C$-object $A$,
\item{(ii)} $\mu \subseteq \G (\mu)$ for every $\C$-morphism $\mu$,
\item{(iii)} if $\nu :\G (A) \to \G (B)$, $\mu :A\to B$ and $\mu \subseteq \nu$ then $\nu =\G (\mu)$  for any $\C$-objects $A,B$ and morphisms $\mu ,\nu$.
\end{defn}
Below we show that any automorphism of a category $\C$ under consideration is a product of two functors the first of which is an inner isomorphism and the second one is an extension functor. 

Let $\Phi$ be an automorphism  of a given category $\C$.  We follow an idea used in \cite{ZhPlVar,ZhPlUnivAlg}. Since $s\sp {\Phi}\sb A :|A|\to |\Phi (A)|$ and $s\sp {\Phi \sp{-1}}\sb A : |A|\to |\Phi \sp {-1} (A)|$ are  $\C$-bijections they determine algebraic structures on the sets $\Im s\sp {\Phi}\sb A=s\sp {\Phi}\sb A |A|) \subseteq |\Phi (A)|$ and $\Im s\sp {\Phi \sp{-1}}\sb A=s\sp {\Phi \sp{-1}}\sb A (|A|) \subseteq |\Phi \sp {-1} (A)|$ correspondingly. Let $\om$ be a symbol of a $k$-ary operation and $\om \sb A$ be the corresponding operation of $A$. Define  new $\om$-operations on $\Im s\sp {\Phi}\sb A $ and $\Im s\sp {\Phi \sp{-1}}\sb A $ by setting for all $a\sb 1,\dots ,a\sb k \in |A|$
$$\om\sb A \sp * (s\sp {\Phi}\sb A (a\sb 1),\dots , s\sp {\Phi}\sb A (a\sb k)) =s\sp {\Phi}\sb A (\om \sb A (a\sb 1,\dots ,a\sb k)),$$
$$\om\sb A \sp \# (s\sp {\Phi \sp{-1}}\sb A (a\sb 1),\dots , s\sp {\Phi \sp{-1}}\sb A (a\sb k)) =s\sp {\Phi \sp{-1}}\sb A (\om \sb A (a\sb 1,\dots ,a\sb k)).$$

The set $\Im s\sb A \sp {\Phi}$ supplied with the operations $\om\sb A \sp *$  for all $\om \in \Om$ is an $\Om$-algebra which we denote by $A\sp *$. It is obvious that the bijective mapping $(s\sb A \sp {\Phi}) \sp*: |A|\to \Im s\sb A \sp {\Phi}$ which is the mapping $s\sb A \sp {\Phi}: |A| \to |\Phi (A)|$ considered as the mapping onto $\Im s\sb A \sp {\Phi}$ is an isomorphism $(s\sb A\sp {\Phi}) \sp*: A\to A\sp *$ . Thus $A\sp *$ and $A$ are isomorphic.  Similarly we have the algebra $A\sp \# $ on the set $\Im s\sp {\Phi \sp{-1}}\sb A$ and the isomorphism $s\sp\#\sb A: A\to A\sp\#$.
\begin{thm}\label{restriction}
Let $\mu : \Phi (A) \to \Phi (B)$ be a homomorphism. The restriction of $\mu$ on $A \sp *$ is a homomorphism  $A \sp * \to B \sp *$. Backwards, any homomorphism $\nu : A \sp * \to B \sp *$ is a restriction of exactly one homomorphism $\mu : \Phi (A)\to \Phi (B)$.
\end{thm}
\begin{proof}
Let $\mu : \Phi (A) \to \Phi (B)$. We have $s\sp {\Phi}\sb B \circ \Phi \sp {-1} (\mu) = \mu \circ s\sp {\Phi}\sb A$. Then 
$ \mu (s\sp {\Phi}\sb A (A)) \subseteq s\sp {\Phi}\sb B (B)$, and it is correct to consider the restriction of $\mu$ on $|A \sp *|$ as a mapping $|A \sp *|\to  |B \sp *|$. Denoting this restriction by $\nu$ we get  $s\sp*\sb B \circ \Phi \sp {-1} (\mu) = \nu \circ s\sp *\sb A$. Consequently, $\nu =s\sp*\sb B \circ  \Phi \sp {-1} (\mu) \circ (s\sp *\sb A)\sp {-1}$ and hence $\nu$ is a homomorphism $ A \sp * \to B \sp *$. 

Backwards, let $\nu : A \sp * \to B \sp *$. Then $\g =(s\sp*\sb B )\sp {-1}\circ  \nu \circ s\sp *\sb A$  is a homomorphism $\g :A\to B$. Let $\mu =\Phi (\g)$. It is obvious that $\mu :\Phi (A) \to \Phi (B)$ and $\nu \subseteq \mu$. The uniqueness of such homomorphism follows from \propref{uniq}.
\end{proof}
The next fact may be useful in some cases.
\begin{prop}\label{condQuasi}
An automorphism $\Phi$ of the category $\C$ is quasi-inner if and only if there exists a central function $(c\sb A)\sb {A\in Ob\C}$ such that for each $ \C$-object $A$ the mapping $c\sb A$ is a bimorphism $A\to \Phi (A)\sp \# $.
\end{prop}
\begin{proof}
Let $\Phi $ be quasi-inner and  $(\st \sb A:A\to \Phi (A)) \sb {A \in Ob \C}$ be the corresponding family of bimorphisms (Definition \ref{inner}). Let $c\sb A =s\sp {\Phi \sp {-1}}\sb {\Phi (A)}\circ \st \sb A$. Obviously $(c\sb A )\sb A$ is a central function. Since the mapping $s\sp {\Phi \sp {-1}}\sb {\Phi (A)}$  can be considered as the isomorphism  $(s\sp {\Phi \sp {-1}}\sb {\Phi (A)})\sp\#: \Phi (A)\to \Phi (A)\sp \# $, we get that  $c\sb A: A\to \Phi (A)\sp \# $ is a bimorphism. 

Now suppose that there exists a central function $(c\sb A) \sb {A \in Ob \C}$ such that every $c\sb A$ is a bimorphism  $A\to \Phi (A)\sp \# $. Let us define $\st \sb A :A\to \Phi (A)$ as 
$$\st \sb A =(s\sb {\Phi (A)}\sp {\Phi \sp {-1}} )\sp{-1} \circ c \sb A =(s\sb {\Phi (A)}\sp {\Phi \sp {-1}\#} )\sp{-1}\circ c \sb A$$ 
for every $\C$-algebra $A$. 
This definition is correct because $c\sb A (|A|)\subseteq |\Phi (A)\sp\#|=\Im s\sb {\Phi (A)}\sp {\Phi \sp {-1}\#}$  and hence   
$\st \sb A :A\to \Phi (A)$ is a bimorphism.

Then we have for every $\mu :A\to B$ that $\mu \circ s\sb {\Phi (A)}\sp {\Phi \sp {-1}}=s\sb {\Phi (B)}\sp {\Phi \sp {-1}}\circ \Phi (\mu)$ which implies  
$$(s\sb {\Phi (B)}\sp {\Phi \sp {-1}})\sp {-1}\circ \mu \circ s\sb {\Phi A)}\sp {\Phi \sp {-1}}\circ (s\sb {\Phi A)}\sp {\Phi \sp {-1}})\sp {-1}\circ c\sb A  =(s\sb {\Phi (B)}\sp {\Phi \sp {-1}})\sp {-1}\circ s\sb {\Phi (B)}\sp {\Phi \sp {-1}}\circ \Phi (\mu)\circ \st \sb A$$
and hence $\st \sb B \circ \mu =\Phi (\mu) \circ \st \sb A$.

Thus $\Phi$ is quasi-inner.
\end{proof}
\begin{rem}
It may be that an automorphism $\Phi$ is quasi-inner but $\Phi \sp {-1}$ is not a such one. It is clear that for $\Phi \sp {-1}$ the statement like \propref{condQuasi}  is true  in which  "*" is used  instead of "\#" .
\end{rem}

Now, we proceed to describing for each $\C$-object $A$ the algebra $A\sp *$, taking into account that  $A\sp *$ perhaps is not an object of $\C$.

Suppose that arities of operations of our variety $V$ are less of some number $n$. Let $F$ be a $\C$-algebra and  $X=\{x\sb 1 ,\dots , x\sb n\}$ be a basis of $F$. Let $\om $ be a symbol of a $k$-ary operation and $\om \sb A$ be the corresponding operation on an algebra $A$. We consider the expression ${\bf w}=\om (x\sb 1 ,\dots , x\sb k)$ as a term in the language corresponding to the our variety. For every $k$-tuple $(a\sb1 ,\dots , a\sb k)$ of elements of an algebra $A$, the value 
$\om \sb A (a\sb 1, \dots ,a\sb k)$ is equal to $\th (\om (x\sb 1, \dots ,x\sb k))$ where $\th : F\to A$ such that $\th (x\sb i)=a\sb i$ for $i=1,\dots, k$.

Now consider the element $\tilde {\bf w}=s\sb A (\bf w)$ of the free algebra $\Phi (F)$.
It turns out that the element $\tilde {\bf w}$ determines an $\om$-operation on the set $s\sb A (|A|)$ for every $\C$-algebra $A$. Let $(u\sb1 ,\dots , u\sb k)$ be a $k$-tuple of elements of $s\sb A (|A|)$. Then we have an unique $k$-tuple $(a\sb1 ,\dots , a\sb k)$ of elements of the algebra $A$ such that $s\sb A (a\sb i) =u\sb i$ for $i=1,\dots, k$. Let $\th : F\to A$ be a homomorphism such that $\th (x\sb i)=a\sb i$ for $i=1,\dots, k$. We define the new operation $\widetilde \om \sb {\Phi (A )}$ by the value of the term  $\tilde{\bf w}$ as follows: 
\begin{equation}\label{derivedoper} 
\widetilde \om \sb {\Phi (A )}(u\sb1 ,\dots , u\sb k)=\Phi (\th) (\tilde {\bf w})
\end{equation}
On the other hand,  have the commutative diagram:
\[
\CD
F @> s\sb F>> \Phi (F)\\
@V\th VV        @VV\Phi (\th) V \\
A @> s\sb A >>\Phi (A)
\endCD
\]
We obtain $u\sb i=s\sb A (a\sb i)=s\sb A(\th (x\sb i)) =\Phi (\th ) (s\sb F (x\sb i))$ and 
\begin{gather*}
\Phi (\th )(\widetilde {\bf w})= \Phi (\th ) \circ s\sb F (\om  (x\sb 1,\dots ,x\sb k))=
s\sb A (\om \sb A (a\sb1, \dots , a\sb k))=\\ = \om \sb A\sp * (s\sb A (a\sb 1),\dots ,s\sb A (a\sb k))=\om \sb A\sp * (u\sb 1,\dots ,u\sb k) .
\end{gather*}
Hence $\widetilde \om\sb {\Phi (A)} (u\sb1, \dots , u\sb k)=\om \sb A\sp * (u\sb1, \dots , u\sb k)$ and the operation $\om \sb A\sp *$ is the derived operation $ \widetilde \om \sb {\Phi (A)}$ on $\Im s\sb A$. By the way this result shows that the operation defined by \eqref{derivedoper} does not depend on the choice of an algebra $F$ having the assumed property. Thus we have proved the following result.

\begin{thm}
For every $\C$-algebra $A$ the derived operation $\widetilde{\omega }\sb \Phi (A) $ defined by \eqref{derivedoper} coincides with the induced operation $\omega \sb A \sp *$, that is,
$$s\sb A (\omega \sb A (a\sb 1 ,\dots ,a\sb k))=\widetilde{\omega }\sb {\Phi (A )}(s\sb A (a\sb 1 ),\dots , s\sb A (a\sb k))$$
for all $a\sb 1,\dots ,a\sb k \in A$.
\end{thm}
Now we are ready to prove the second main theorem.
\begin{thm}\label{product}
Let $\C$ be a  full subcategory of $\Th \sp 0 (\V)$  containing  a monogenic algebra $A\sb 0$. Let $\Phi$ be an automorphism of $\C$ . Then 
\item {(i)} there is a full subcategory $\D$ of $\Th\sp 0 (\V)$ such that $\Phi$  is a product $\Phi =\G \circ \Psi$, where $\Psi : \C \to \D$ is an inner isomorphism and $\G : \D \to \C$ is an extension functor. 
\item {(ii)} there is a term ${\bf t}(x\sb 0) \in |A\sb 0|$ such that  $|\Psi (A)|=\{ {\bf t}(u)\vert u\in |\Phi (A)\}|$ for every $\C$-object $A$.
\end{thm}
\begin{proof}
Let $\D$ be the full subcategory of $\Th \sp 0 (\V)$ whose objects are all algebras $A\sp *$ where $A$ is an arbitrary $\C$-object. Let $\Psi (A)=A\sp *$ for every $\C$-object $A$ and $\Psi (\mu) =s\sp *\sb B \circ \mu \circ (s\sp * \sb A)\sp {-1}$ for every $\C$-morphism $\mu :A\to B$. It is clear that $\Psi : \C \to \D $ is a functor and, what is more, $\Psi$ is an inner isomorphism between $\C$ and $\D$. 

Now let $\G =\Phi \circ \Psi \sp{-1}$. Then $\G$ is a functor $\D\to \C$. Of course, $\G$ is an isomorphism between these categories and $\Phi =\G \circ \Psi$. We have $\G (A\sp *) =\Phi (A)$ and therefore $|A\sp *| \subseteq |\G (A\sp *)|$. 
Let $\mu :A\sp* \to B\sp *$. According to the definition of $\Psi$, we have $\nu =\Psi \sp{-1}(\mu) =(s\sp *\sb B)\sp {-1} \circ \mu \circ s\sp * \sb A$ is a homomorphism $A\to B$. Thus $\G (\mu ) =\Phi (\nu)$. On the other hand, 
$s\sb B \circ (s\sp *\sb B)\sp{-1} \circ \mu \circ s\sp * \sb A \circ (s\sb A)\sp {-1}=\mu$, that is, $\mu$ is the restriction of $\G (\mu)$ on $A\sp*$. In view of Theorem \ref{restriction} and Definition \ref{functors} we obtain that $\G$ is an extension functor. The second statement follows from \thmref{image}.
\end{proof}
\begin{cor}\label{Corollary}
The last theorem shows that the process of describing of an automorphism $\Phi$ of a category $\C$ under consideration is reduced to the following steps:
\begin{enumerate}
\item  We view the elements of $A\sb 0$ and find the general form of the term ${\bf t}(x\sb 0) \in |A\sb 0|$. Then  
we  describe the subset $s\sb A (|A|)$ of an algebra $\Phi (A)$ according  to the formula $s\sb A (|A|)=\{ {\bf t}(u): u\in |\Phi (A)|\}$.  In the case ${\bf t}(x\sb 0) =x\sb 0$ this set is equal to the underlying set of $\Phi (A)$.
\item In order to describe operations of $A\sp *$, we use the fact that the restriction of any endomorphism of $\Phi (A)$ to $s\sb A (|A|)$ must be an endomorphism of the algebra $A\sp *$, and vice versa, any endomorphism of $A\sp *$ could be extended up to an endomorphism of $\Phi (A)$ which is uniquely determined. 
\item We use the fact that the correspondence described above is an isomorphism between semigroups $\End (A\sp *)$ and $\End (\Phi (A))$.
\end{enumerate}
If the requirements of these steps are fulfilled, it remains to describe the kind of embedding of $A\sp *$ in $\Phi (A)$, which may be an isomorphism or some new kind of a correspondence, for example, a mirror homomorphism or a screw-homomorphism.
\end{cor}
\begin{examples*}
%example 1
\item{1.} First we will show how method suggested above can be applied in a simple case when the result is already known, namely for the category $SEM$ of all  finitely generated free semigroups. Let $\C$ be a full subcategory of $SEM$ containing a monogenic semigroup $A\sb 0$. Let $\Phi$ be an automorphism of $\C$ . Any term ${\bf t}(x\sb 0) \in A\sb 0$ in our case has a form ${\bf t}(x\sb 0) =x\sb 0 \sp k$ where $k\geq 1$ is a whole number. Let $ F=F(x\sb 1,\dots , x\sb n)$ be a free semigroup generated by variables $x\sb1,\dots , x\sb n$ and $A=\Phi \sp{-1}( F)$.  Thus $A\sp *=\{u\sp k\vert u\in F\}$. Let $||w||$ denote the length of the word $w\in F$.

Let $y$ be an element of a basis of $A\sp *$, then there exists an endomorphism $\g$ of $A\sp *$ such that $x\sb 1\sp k =\g (y)$. Let the endomorphism $\tilde{\g}$ of  $ F$ be the extension of $\g$. Thus $\tilde{\g} (y)=x\sb 1\sp k$. 
Since $y=u\sp k$ for some $u\in F$ we get $k\geq ||y||\geq k$ and hence $y\in \{x\sb 1\sp k ,\dots , x\sb n\sp k\}$. Applying this result to $A\sb 0$ (n=1) we obtain that $(\Phi \sp {-1}(A\sb 0))\sp *$  and $A\sb 0$ are isomorphic and hence $\Phi \sp {-1}(A\sb 0)$ and $A\sb 0$ are isomorphic. We know (\secref{Free}) that in this case all mappings $s\sb A$ are surjective and hence $k=1$.
Therefore $(x\sb 1, x\sb 2)$ is the common basis of semigroups  $ F=F(x\sb 1, x\sb2)$ and $A\sp *$ which are isomorphic. 

Since  $(x\sb 1, x\sb 2)$ is the unique basis of $F$ and the unique basis of $A\sp*$, there are only two isomorphisms  $A\sp* \to F$, namely $\varphi (x\sb 1) =x\sb 1 , \varphi (x\sb2 ) = x\sb 2$ or $\varphi (x\sb 1) =x\sb 2, \varphi (x\sb 2 )=x\sb1$. In the former case we have $\varphi (x\sb 1 *x\sb 2) =x\sb 1x\sb 2 $ and hence $\varphi$ is the identity mapping. In
the latter case, we have $\varphi (x\sb 1 *x\sb 2) =x\sb 2x\sb 1 $ and therefore $\varphi$ maps every word $u$ to the word $\underline{u}$, where all letters are written in the reverse order.  We obtain that $\Phi$ is produced by isomorphisms, that is, $\Phi$ is inner, or all mappings $s \sb A : A\to \Phi (A)$ are anti-isomorphisms.
%example 2
\item{2.}
Now we apply our method to the variety of modules. As far as the author knows there are no results in the general case. 
There are some essential results in this topic in \cite{KLP} and \cite{Lipyan}. 
Let $ R$ be an arbitrary ring with unit and ${\bf Mod - R }$ denote the category of all finitely generated free left $R$-modules. We consider a full subcategory $\C$ of the category ${\bf Mod - R }$ which satisfies the accepted conditions. In this case $A\sb 0  =Rx\sb 0$.

Let $\Phi$ be an automorphism of $\C$.  Consider a free left $R$-module ${\bf M} =(M, +, 0, F)$  where $F: R\times M \to M$  is the left action of $R$ on $M$. Assume that    $\{x\sb 1 ,x\sb 2\}$ is a basis of ${\bf M}$, that is, $M = Rx\sb 1 \oplus Rx\sb 2$. Let $A=\Phi \sp {-1} ({\bf M})$. We have to describe the finite generated free left $R$-module $A\sp *=(|A\sp *|,+\sp *, 0\sp *, F\sp *)$.  Since any term ${\bf t}(x\sb 0) \in A\sb 0$ in our case has a form ${\bf t}(x\sb 0) =rx\sb 0 $ where $r\in R ,\; r\not = 0$, we get $|A\sp *|=rM$. It is clear that $0\in |A\sp*|$.

On the other hand, for every endomorphism $\varphi$ of $A\sp *$ we have $\varphi (0\sp *)=0\sp *$. Then according to (2) in \corref{Corollary} $\g (0\sp *)=0\sp *$ for every endomorphism $\g$ of $\bf M$. Thus $0\sp * =0$. Further  $r x\sb 1 +\sp* r x\sb 2 =i x\sb 1 +j x\sb 2 $ for some $i,\; j \in R$.
 Let $\g :{\bf M}\to {\bf M} $ such that $\g ( x\sb 1)=x\sb 1,\;\g ( x\sb 2) =0$. Since restriction of $\g$ to $|A\sp*|$ is an endomorphism of $A\sp*$ too, we get 
$$\g (r x\sb 1 +\sp* r x\sb 2)=\g (r x\sb 1) +\sp* \g (r x\sb 2) =r\g (x\sb 1)+\sp* r\g ( x\sb 2)=r x\sb 1+\sp* 0=rx\sb 1.$$
On the other hand, $\g (i x\sb 1 +j x\sb 2) =ix\sb 1$. Thus $rx\sb 1=ix\sb 1$ and hence $i=r$. In exactly the same way we get $j=r$ and hence $r x\sb 1 +\sp* r x\sb 2 =r x\sb 1 +r x\sb 2 $ . 

Let $a\sb1 , a\sb 2 \in M$ and $\g$ be an endomorphism of $\bf M$ such that $\g (x\sb1)=a\sb 1$ and $\g (x\sb 2)=a\sb 2$.
Just as above we get
$$ \g (r x\sb 1 +\sp* r x\sb 2)=\g (r x\sb 1)+\sp* \g (r x\sb 2)=r\g (x\sb 1)+\sp* r\g (x\sb 2)=ra\sb1+\sp* ra\sb 2.$$
On the other other hand we get
$$\g (r x\sb 1 + r x\sb 2)=\g (r x\sb 1) +\g (r x\sb 2)=r\g (x\sb 1) +r\g(x\sb 2)=ra\sb 1+ra\sb 2 .$$
As a result we obtain
$$ra\sb1+\sp* ra\sb 2 =ra\sb 1+ra\sb 2,$$
which leads after obvious calculations to the fact that operations $+\sp *$ and $+$ coincide on $|A\sp*|$. Thus the different between structures $\bf M$ and $ A\sp *$ may be only in actions of the ring $R$.  

It is obvious that for every $k\in R$ there is an element $k\sp*\in R$ such that $k*rx\sb1=rk\sp* x\sb1$. Such an element 
$k\sp*$ may be is not uniquely determined.  But there exists a function $\a :R\to R$ such that $k*rx\sb1=r\a (k)x\sb1$. For  two such function $\a\sb 1$ and $\a\sb 2$ we have $r\a\sb 1 =r\a\sb 2$. In the same way as above we get: $k*ra=r\a (k)a$ for every $a\in M$. Further 
$$( k\sb 1 k\sb 2)* ra =k\sb 1 * (k\sb 2 *ra)=k\sb 1* r\a (k\sb 2)a =r(\a (k\sb 1)\a (k\sb 2))a .$$
On the other other hand we get
$$ ( k\sb 1 k\sb 2)* ra=r\a (k\sb 1 k\sb 2)a.$$
Since $a$ is an arbitrary element of $M$ we obtain that 
\begin{equation}\label{function1}
r\a (k\sb 1 k\sb 2)=r(\a (k\sb 1)\a (k\sb 2)).
\end{equation}
In the same way we get that 
\begin{equation}\label{function2}
r\a (k\sb 1+k\sb 2) =r\a (k\sb 1)+r\a (k\sb 2)
\end{equation}
and
\begin{equation}\label{function3}
r\a (1)=r.
\end{equation}
Summing up these investigations, we obtain the following description of automorphisms of $\C$. 

For every automorphism $\Phi$ of $\C$ the main function $(s\sb A: A \to \Phi (A))\sb {A\in Ob\C}$ satisfies the following conditions:
\item (1) there exists an element $r\in R$ such that for every $a\in A$ $s\sb A (a)=ru$ for some $u\in \Phi (A)$, 
\item (2) every $\C$-bijection $ s\sb A: A \to \Phi (A)$ is an additive mapping,
\item (3) there exist a function $\a : R\to R$ such that for every $k\in R$ and every $ a\in A $ it holds $s\sb A (ka)=r \a (k)u$ where $ru=s\sb A (a), u\in \Phi (A)$,
\item (4) the function $\a $  satisfies the equations \ref{function1} - \ref{function3}. 

Consider the case when the ring $R$ does not contain zero divisors. Let $B=\Phi\sp {-1}(Rx\sb 0)$.  Suppose that a basis of the module $B\sp *$ contains two different elements $y\sb 1$ and $y\sb 2$. Let $\g$ be the automorphism of $B\sp *$ such that $\g (y\sb 1) =y\sb 2$ and $\g (y\sb 2) =y\sb 1$ and $\tilde {\g}$ is the extension of $\g$ up to an automorphism of $Rx\sb 0$. Since $\g \sp 2 =1\sb {B\sp *}$ the same property is valid for $\tilde {\g}$. Therefore  we get  $\tilde{\g}\sp 2 (x\sb 0) =x\sb 0$. Since $\tilde {\g}(x\sb 0) =kx\sb 0$ for some $k\in R$ we get that $k\sp 2 =1$ and hence $k=1$ or $k=-1$. Thus $y\sb 1 = y\sb 2$ or $y\sb 1 =- y\sb 2$ that is contrary to the assumption. 

We obtain that $\Phi\sp {-1}(Rx\sb 0)$  and $(Rx\sb 0)$ are isomorphic. In this case all mappings $s\sb A :A\to \Phi (A) $ are surjective and $r=1$.  Therefore all automorphisms of $\C$ are semi-inner, that is, the mappings $s\sb A : A\to \Phi (A)$ are additive bijections and there exists an endomorphism $\a$ of $R$ such that $s\sb A (ka)=\a (k) s\sb A (a)$ for all $k\in R$ and $a\in A$. Since the same is true for the automorphism $\Phi \sp {-1}$ we conclude that $\a$ is an automorphism of the ring $R$.
\end{examples*}
{\bf Acknowledgments}
The author is pleased to thank B. Plotkin and E. Plotkin, R. Lipjansky and G. Mashevitsky for useful discussions and interesting suggestions.

\end{document}